\documentclass[10pt]{amsart}
\usepackage{amsfonts}
\usepackage{amssymb}
\usepackage{amsmath}
\usepackage[arrow,matrix,curve]{xy}

\usepackage[active]{srcltx}

\usepackage[normalem]{ulem}
\usepackage{color}

\newcommand{\C}{{\mathbb C}}

\newcommand{\F}{{\mathbb F}}

\newtheorem{theorem}{Theorem}[section]
\newtheorem{lemma}[theorem]{Lemma}

\newtheorem{corollary}[theorem]{Corollary}
\newtheorem{proposition}[theorem]{Proposition}
\newtheorem{definition}[theorem]{Definition}

\def\cal{\mathcal}

\newcommand{\call}[0]{{\cal L}}
\newcommand{\calg}[0]{{\cal G}}

\newcommand{\cala}[0]{{\cal A}}

\newcommand{\calm}[0]{{\cal M}}

\begin{document}

\title[A Green imprimitivity theorem]{An elementary Green imprimitivity theorem for inverse semigroups}
\author[B. Burgstaller]{Bernhard Burgstaller}
\address{Departamento de Matematica, Universidade Federal de Santa Catarina, CEP 88.040-900 Florian\'opolis-SC, Brasil}
\email{bernhardburgstaller@yahoo.de}

\begin{abstract}
A Morita equivalence similar to that found by Green for crossed products by groups will be established for crossed products by inverse semigroups.
More precisely,
let $G$ be an inverse semigroup, $H$ a finite sub-inverse semigroup of $G$ and $A$ a $G$-algebra or a $H$-algebra.
Then the crossed product $A \rtimes H$ is Morita equivalent
to a certain crossed product $B \rtimes G$.
\end{abstract}

\maketitle

\section{Introduction}

In a classical paper \cite{green1978}, Green showed that for a closed subgroup $H$ of a locally compact group $G$,
and a $G$-algebra $A$ there exits a Morita equivalence between $A \rtimes H$ and
$C_0(G/H,A) \rtimes G$ via an imprimitivity bimodule over these algebras (\cite[Prop. 3]{green1978}).
This useful result was discussed and generalized in many directions, for example, in
\cite{raeburn1988, vaes2005, huefraeburnwilliams2003, echterhoffkaliszewski2006}. 


In this note we shall
establish an analogous imprimitivity theorem 
for
an inverse semigroup $G$ and a finite sub-inverse semigroup $H \subseteq G$ 
for crossed products in Sieben's sense \cite{sieben1997}.
As a corollary of this, we show this holds true also for a given $H$-algebra $A$, and thus
this may be usefully combined with induction like in Kasparov \cite{kasparov1995,kasparov1988}.
Actually, this note was motivated by the fact that
the Baum--Connes map \cite{baumconneshigson1994} for groups $G$ is a kind of extrapolation of Green--Julg isomorphisms
for crossed products by $G$ of induced algebras by compact subgroups $H \subseteq G$,
as noted by Meyer and Nest in \cite{meyernest2006}.
In establishing that, Kasparov's induction plays a fundamental role. 
To potentially carry 
this result  
over from groups to inverse semigroups,
we need induction for compact (and thus finite) sub-inverse semigroups $H \subseteq G$,
and this is now provided in this note. 
Actually, in the meanwhile we have made considerable progress in this direction
and were able to establish a Baum--Connes map for fibered $G$-algebras
\cite{burgAttempts,burgNoteBC} founding on this note.

We are going to give a brief summary of 
this article. 
At first we rewrite the inverse semigroup crossed product $A \rtimes H$
as a groupoid crossed product $A \rtimes \calg$ to have a group-like construction.
Then we adapt and follow Green's proof \cite[p. 199-204]{green1978} in a natural way. The action on a certain quotient space $G_\calg/\calg$
($G/H$ in Green \cite{green1978})
is similar to the regular representation action by Khoshkam and Skandalis \cite{khoshkamskandalis2004}.
After establishing Green's imprimitivity Theorem \ref{theoremImprimitivity}, we apply it to the induced algebra
(in the sense of Kasparov \cite{kasparov1995,kasparov1988}) $A$ of a $H$-algebra $D$, and restrict to ideals
to get the second Green imprimitivity theorem, Corollary \ref{greenImp2}.




\section{Preparing definitions and crossed products}

We begin by recalling crossed products in the sense of Khoshkam and Skandalis \cite{khoshkamskandalis2004} and Sieben \cite{sieben1997},
but use several notions from \cite{burgiDescent}. Let $G$ denote an inverse semigroup.

\begin{definition}
{\rm
A {\em $G$-algebra} $A$ is a $C^*$-algebra $A$ endowed with a $G$-action in the following sense:
there exists a semigroup homomorphism $\alpha: G \rightarrow \mbox{End}(A)$, written as $g(a):= \alpha_g(a)$, such that
\begin{equation}  \label{sc2}
g g^*(a) b = a g g^*(b)
\end{equation}
for all $a,b \in A$ and $g \in G$.
}
\end{definition}

Such a $G$-algebra (whose definition is equivalent to \cite[Def. 3.1]{burgiDescent}) is a special case of $G$-algebras in the sense of \cite{sieben1997} and \cite{khoshkamskandalis2004}.



\begin{definition}
{\rm
Let $\F(G,A)$, or $\F$ for brevity, be the universal $*$-algebra over $\C$ generated by disjoint copies of $A$ and $G$ 
subject to the relations that
the $*$-algebraic relations of $A$ are respected, the multiplication and involution of $G$ are respected, and
the relations
\begin{equation}   \label{definingrelations}
g(a) g g^* = g a g^*, \quad g g^* a = a g g^*
\end{equation}
hold true for all $a \in A$ and $g \in G$.
We shall identify $G$ and $A$ as subsets of $\F$.
The {\em algebraic crossed product} $A \rtimes_{alg} G  \subseteq \F$ denotes the linear span of all elements of the form $a g$
($a \in A, g \in G$), which are usually denoted by $a \rtimes g$, and is a $*$-subalgebra of $\F$.
}
\end{definition}

\begin{definition}
{\rm
We denote by $G_0 \subseteq G$ the idempotent elements of $G$, and by $E(G) \subseteq \F$ the set of all projections
of the form $e_0 (1-e_1) \ldots (1-e_n) \in \F$ with $e_0,\ldots , e_n \in G_0$ and $n \ge 0$.
The subset $G_E:= \{g p \in \F\,|\, g \in G ,  p \in E(G)\} \subseteq \F$ is an inverse semigroup in $\F$ (under multiplication and  taking adjoint).
We even shall write $a \rtimes g:=a g$ when $a \in A$ and $g \in G_E$.
}
\end{definition}

Note that then the identities
\begin{equation}   \label{sc1}
a \rtimes g p = a g p = g g^* a g g^* g p =  g g^*(a) g p
=  g g^*(a) \rtimes g p
\end{equation}
hold in $\F$ for all $a \in A$ and $g p \in G_E$ ($g \in G, p \in E(G)$).

\begin{definition}
{\rm
Hence, it is natural to call 
an expression $a \rtimes g p \in \F$ with $g \in G,p \in E(G)$ and 
$a \in A_{g g^*} := g g^*(A)$ 
to be {\em standard}.
}
\end{definition}


The reader should be cautioned that the first relation of (\ref{definingrelations})
is {\em not} true in general for $g \in G_E$
(consider for example $(1-e)(a)=0$ for the trivial $G$-action on $A$), however the second relation of (\ref{definingrelations})
and identity (\ref{sc2}) remain true for $g \in G_E$ (see Lemma \ref{lemma1} below).

The {\em full crossed product} $A \rtimes G$ is the closure of the image of $A \rtimes_{alg} G$ under the universal $*$-representation $\pi$ of $\F$ on Hilbert space (\cite[Def. 5.4]{khoshkamskandalis2004} or \cite[5.16, 6.2, 8.4]{burgiDescent}).  
It is easy to see with the reduced representations \cite[p. 271]{khoshkamskandalis2004} that
$\pi$ is injective on $A \rtimes_{alg} G$, and so the latter is a pre-$C^*$-algebra
with a $C^*$-norm.
Sieben's crossed product $A \widehat \rtimes G$ is defined to be the image of $A \rtimes_{alg} G$
under the universal $*$-representation $\tau$ of $\F$ on Hilbert space satisfying
$\tau(g(a) - g a g^*) = 0$ (see \cite{sieben1997}).
We write $a \widehat \rtimes g$ for $\tau(a g)$.
Note, in particular,
that $\widehat \rtimes$ is compatible:
\begin{equation} \label{compat}
e(a) \widehat \rtimes g \;=\; a \widehat \rtimes e g  \qquad \forall a \in A,\, g \in G_E, \,e \in E(G).
\end{equation}
Notice that this identity is not true for $\rtimes$, and this compatibility
is actually the essential difference between
the full crossed product and Sieben's crossed product.

\begin{definition}   \label{defH}
{\rm
Let us now be given a finite sub-inverse semigroup $H' \subseteq G$ of $G$. Denote by $H$ the (finite) groupoid associated to $H'$ (cf. \cite{0913.22001}).
More precisely, let $H^{(0)} \subseteq \F$ be the set of all nonzero {\em minimal} projections of $E(H')$ and
$$H \; =\; \{ t e \in \F\,|\, t \in H', e \in H^{(0)}\} \backslash \{0\} \quad \subseteq \F.$$
The multiplication within $H$ is that inherited from $\F$.
}
\end{definition}

\begin{definition}
{\rm
Define
$$G_H \;=\; \{ g e \in \F\,|\, g \in G, e \in H^{(0)}, g^* g \ge e\} \backslash \{0\} \quad \subseteq \F.$$
We endow $G_H$ with an equivalence relation:
$g \equiv h$ if and only if there exists $t \in H$ such that $g t = h$ ($g,h \in G_H$).
We denote by $G_H/H$ the set-theoretical quotient of $G_H$ by $\equiv$.
}
\end{definition}

We shall exclusively work with representatives in this quotient; writing $g \in G_H/H$ means implicitly that $g \in G_H$
and we use no class brackets;
if then $g \in G_H$ is meant or the class $g \in G_H/H$ becomes apparent from the context.
For an assertion $\cala$ we let $[\cala]$ be the real number $0$ if $\cala$ is false, and $1$ if $\cala$ is true.

\begin{definition}
{\rm
Let $C_0(G_H/H)$ denote the commutative $C^*$-algebra of (continuous) complex-valued functions vanishing at infinity of the (discrete) set $G_H/H$
with the pointwise operations. 
The delta function $\delta_g$ in $C_0(G_H/H)$ is denoted by $g$ ($g \in G_H/H$).
The algebra $C_0(G_H/H)$ is endowed with the $G$-action $g(h):= [g h \in G_H] \,\, g h$, where $g \in G$ and $h \in G_H/H$
(of course, $g h \in G_H$ is equivalent to $g^* g \ge h h^*$).
We let $A \otimes C_0(G_H/H)$ be the $C^*$-algebraic tensor product endowed with the diagonal action by $G$.
}
\end{definition}

\begin{lemma}  \label{lemma1}
\begin{itemize}

\item[(i)]
If $g_1, \ldots,g_n \in G_H$ are mutually different 
then
$\sum_{i=1}^n a_i \rtimes g_i = 0$ (sum of standard elements) implies $a_1 = \ldots = a_n = 0$.

\item[(ii)]
The $G$-action on a $G$-algebra $A$ extends naturally to an inverse semigroup $G_E$-action on $A$
(i.e. one sets $(1-e)(a):= a - e(a)$ for all $a \in A$ and $e \in E$).

\item[(iii)]
The formulas $(a \rtimes g) (b \rtimes h) = a g(b) \rtimes g h$ and $(b \rtimes h)^* = h^*(b^*) \rtimes h^*$
hold in $\F$ for all $g, h \in G_E$, $a \in A_{g g^*} := g g^*(A)$ and $b \in A$.

\end{itemize}

\end{lemma}

\begin{proof}
(i) 
We may assume that all $g_i$ have the same source projection in $H^{(0)}$
(otherwise multiply $\sum a_i \rtimes g_i$ from the left with each single minimal mutually orthogonal 
projection of $H^{(0)}$).
Hence we may fix $e_{1},\ldots,e_m \in G_0$ such that for all $1 \le i \le n$ we have $g_i= h_i  (1- e_1) \ldots (1- e_m)$
for certain mutually different $h_i \in G$.
Expanding, we get $0=\sum_{i=1}^n a_i g_i = \sum_{i=1}^n a_i h_i - \sum_{i=1}^n a_i h_i e_{1} \pm \ldots$ in $\F$.
Now note that by the reduced representation of the algebraic crossed product $A \rtimes_{alg} G$ in \cite[p. 271]{khoshkamskandalis2004} the 
elements $h_i, h_j e_1$ et cetera in the
last sum are linearly independent, as far as they are all different. 
Certainly, however, we may conclude
that $a_i h_i = 0$ for all $1 \le i \le n$,
because assuming that $h_i = h_j e_{k_1} \ldots e_{k_s}$ would yield $g_i = h_j e_{k_1} \ldots e_{k_s} (1-e_1) \ldots (1-e_m) =0$
(however $0 \notin G_H$).
This yields the claim.

(ii)
By the linear independence of the elements of $G_0$,
it easy to see that we have already a well-defined semigroup homomorphism $\alpha: E(G) \rightarrow \mbox{End} (A)$
defined by $\alpha_e(a)= e(a)$ and $\alpha_{1-e}(a) = a - e(a)$ for all $a \in A$ and $e \in G_0$. 
%
To extend it to $G_E$, we consider an ambiguous representation
$0 \neq g p = h q \in G_E$ for some $h,g \in G$ and $p,q \in E(G)$. Then
$g p q = h p q$ and so $g = h$ by a similar argument as in (i). Thus $g^* g p = g^* g q$ in $E(G)$.
Hence, the definition $\alpha_{g p}:= \alpha_g \alpha_p = \alpha_g \alpha_{g g^* p} = \alpha_{h} \alpha_q$ is well-defined ($\alpha_g$ denotes the given $G$-action).

(iii) Let $g,h \in G$, $p,q \in E(G)$, $a \in A_{g p g^*}$ and $b \in A$. We have, for example, by (\ref{sc2})
for the extended action of (ii), (\ref{definingrelations}) and because $a = g p g^*(a) \in A_{g p g^*}$ that
$$a g p \cdot b h q = g p g^*(a) \cdot g p \cdot b \cdot h q = g p g^*(a) \cdot g(b) \cdot g p\cdot h q =
 a \cdot g p(b) \cdot g p h q$$
%
in $\F$.
\end{proof}


\begin{lemma}   \label{lemma2}
%
If $A$ is a $G$-algebra and $I \subseteq A$ a $G$-invariant ideal in $A$ then $I \rtimes G \subseteq A \rtimes G$
and $I \widehat \rtimes G \subseteq A \widehat \rtimes G$ canonically.
\end{lemma}

\begin{proof}
A representation of $\F(G,I)$ is given by a covariant triple $(\sigma,U,H)$ for some Hilbert space $H$, a $*$-homomorphism
$\sigma: I \rightarrow B(H)$ and an inverse semigroup homomorphism $U: G \rightarrow B(H)$ satisfying the analogous defining relations as in
(\ref{definingrelations}).
Proceed as in
\cite[Lemma A.4]{baumguentnerwillett} 
to show that the representation of $\F(G,I)$ extends to $\F(G,A)$.
%
Hence $\overline{\F(G,I)} \supseteq I \rtimes G \rightarrow A \rtimes G \subseteq \overline{\F(G,A)}$ is isometric.
\end{proof}

\section{The imprimitivity theorem}

We shall now 
introduce an imprimitivity bimodule in the sense of Rieffel \cite{rieffel1974}.

\begin{definition}   \label{defimprimi}
{\rm
We introduce the spaces  
\begin{eqnarray*}
B_0 &=& A \rtimes_{alg} H  := \mbox{span} \{ a\rtimes t \in A \rtimes_{alg} G|\, a \in A_{t t^*}, \, t \in H\}, \\
X_0 &=& \mbox{span} \{ a \rtimes g \in A \rtimes_{alg} G|\, a \in A, g \in G_H\}, \\
E_0 &=& \big( A \otimes C_0(G_H/H) \big) \rtimes_{alg} G .
\end{eqnarray*}
The spaces $B_0 \subseteq A \rtimes G$ and $E_0$ are regarded as pre-$C^*$-algebras.
We make $X_0$ to a right pre-Hilbert module over $B_0$ (cf. \cite[Def. 2.8]{rieffel1974}) by the following operations
\begin{eqnarray*}
&&X_0 \times B_0 \longrightarrow X_0: \quad (a \rtimes g)(c \rtimes t) \,\,:=\,\, a g(c) \rtimes g t ,  \\
&&X_0 \times X_0 \longrightarrow B_0: \quad \langle a \rtimes g, b \rtimes h \rangle_{B_0} \,\,:= \,\, [g^* h \in H] \,\,g^*(a^* b) \rtimes g^* h
\end{eqnarray*}
for $a,b \in A$, $c \in A_{t t^*}$, $g,h \in G_H$ and $t \in H$,
and to a left pre-Hilbert module over $E_0$ by
\begin{eqnarray*}
&&E_0 \times X_0 \longrightarrow X_0: \quad
(a \otimes r \rtimes s) (b \rtimes j) \,\,  := \,\, \, [s j \in G_H] \, [r \equiv s j]  \,\, a  s(b) \rtimes s j  , \\
&&X_0 \times X_0 \longrightarrow E_0: \quad
\langle a \rtimes g, b \rtimes h \rangle_{E_0} \,\,:=   \,\,  
a \, gh^*(b^*) \otimes  g \rtimes g h^*
\end{eqnarray*}
for $a,b \in A$, $r \in G_H/H$, $s \in G$ and $j, g, h \in G_H$.
}
\end{definition}

Because of identity (\ref{sc1}) we may write every element in an algebraic crossed product
as the sum of standard elements.
Because of the linear independence statement of Lemma \ref{lemma1}.(i)
we may then extend the formulas of Definition \ref{defimprimi} for standard expressions
by linearity.
We need however remark, that


\begin{lemma}
The formulas of Definition \ref{defimprimi}
remain however valid also for non-standard expressions as stated.
\end{lemma}

\begin{proof}
For example, by considering the 
formula of the map $E_0 \times X_0 \rightarrow X_0$,
given elementary elements $a \otimes r \rtimes s s^*s \in E_0$ and $b \rtimes j p \in X_0$
with $a \in A, r \in G_H/H, s,j \in G$ and
$s^* s, p \in E(G)$,
we go over to their
standard form 
$[s s^* r \in G_H] s s^*(a) \otimes s s^* r \rtimes s \in E_0$ and $j j^*(b) \rtimes j p \in X_0$
according to (\ref{sc1}).
Then their module product in $X_0$ is
\begin{eqnarray*}
&& [s s^* r \in G_H] \, [s j p \in G_H] \, [ s s^* r \equiv s j p]  \,\, s s^*(a)  s( j j^*(b)) \rtimes s j p \\
&=& [s j p \in G_H] \, [r \equiv s j p]  \,\, s j j^* s^* (a)  \, s j j^* s^* s(b) \rtimes s j p   \\
&=& [s j p \in G_H] \, [r \equiv s j p]  \,\, a s(b) \rtimes s j p
\end{eqnarray*}
by (\ref{sc2}),  
(\ref{sc1})
and because $[s s^* r \in G_H]$ cancels because $s s^* r \equiv s j p$ implies $s s^* r = s j p t$ for some $t \in H^{(0)}$, 
implies $s s^* r = r$ 
(since the source projection of $s s^* r = s j p t \in G_H$ is in $H^{(0)}$ and thus cannot become smaller than the one
of $r \in G_H$),
implies $s s^* r = r \in G_H$. 
This is the same element
as taking the module product formula in $X_0$ for the given non-standard elements 
$a \otimes r \rtimes s \in E_0$ and $b \rtimes j p \in X_0$.
%
%
The above formulas for non-standard expressions can then be extended also linearly.
%
%
%
\end{proof}

\begin{proposition}   \label{straightprop}
Straightforward, but again rather time-consuming, extensive and tedious computations show that we have
\begin{eqnarray}
&&\langle x,y b \rangle_{B_0} = \langle x,y \rangle_{B_0} b , \quad \langle x,y \rangle_{B_0}^* = \langle y,x \rangle_{B_0},  \nonumber \\
&&\langle f x,y \rangle_{E_0} = f \langle x,y \rangle_{E_0} , \quad \langle x,y \rangle_{E_0}^* = \langle y,x \rangle_{E_0},  \nonumber  \\
&&\langle f x,y \rangle_{B_0} =  \langle x, f^* y \rangle_{B_0} , \quad \langle x, y b \rangle_{E_0} = \langle x b^*,y \rangle_{E_0},
\quad x \langle y,z \rangle_{B_0} = \langle x,y\rangle_{E_0} z  \label{eq8}
\end{eqnarray}
for all $x,y,z \in X_0$, $b \in B_0$ and $f \in E_0$ (cf. \cite[Def. 6.10]{rieffel1974}).
%
\end{proposition}

\begin{proof}
For convenience of the reader we demonstrate 
the first identity of line (\ref{eq8}), which is the most
difficult of all of these, and the 
others should not present further new difficulties (for the $*$-algebraic operations in $B_0$
use the formulas of Lemma \ref{lemma1}.(iii)). 
We have
\begin{eqnarray}
&& \big \langle (a \otimes r \rtimes s) ( b \rtimes g), c \rtimes h \big  \rangle_{B_0}  \nonumber \\
&=& [s g \in G_H] \, [r \equiv s g]  \,\, \big \langle
 a  s(b) \rtimes s g 
, c \rtimes h \big  \rangle_{B_0}  \nonumber \\
&=&   [s g \in G_H] \,\,[r \equiv s g]  \,\, [g^* s^* h \in H]  \,\, g^* s^* \big ( s(b^*) a^* c \big )   \rtimes g^* s^* h   \label{x4}
\end{eqnarray}
for $a,b,c \in A$, $r \in G_H/H$, $s \in G$ and $g,h \in G_H$, and
\begin{eqnarray}
&& \big \langle  b \rtimes g, (a \otimes r \rtimes s)^* (c \rtimes h) \big \rangle_{B_0}  \nonumber \\
&=& \big \langle  b \rtimes g, \big([s^* r \in G_H] \,\,s^*(a^*) \otimes s^* r \rtimes s^* \big) (c \rtimes h) \big \rangle_{B_0} \nonumber \\
&=& [s^* r \in G_H] \,\, \, [s^* h \in G_H] \, [ s^* r \equiv s^* h]  \,\, 
\big \langle  b \rtimes g, 
 s^*( a^*)  s^*(c) \rtimes s^* h  \big \rangle_{B_0} \nonumber \\
&=& [g^* s^* h \in H] \,\, [s^* r \equiv s^* h] \,\,[s^* h \in G_H] \,\,[s^* r \in G_H] \,\,g^*\big (b^* s^*(a^*) s^*(c) \big) \rtimes g^* s^* h . \label{x5}
\end{eqnarray}

We only need to show that the scalar coefficients of the expressions 
(\ref{x4}) and (\ref{x5}) coincide, because observe that the vector parts
coincides by a single application of identity (\ref{sc2}).
Note that the scalar $[g^* s^* h \in H]$ appears both in (\ref{x4}) and (\ref{x5}).
Suppose that (\ref{x4}) is nonzero. Then
$sg \in G_H$, and thus $s^* s \ge g g^*$ 
because $g \in G_H$ and $s g \in G_H$ (the source projection of $g$ is in $H^{(0)}$ and cannot be made smaller in $s g$).
On the other hand, $r \equiv sg$ and so there exists some $t \in H$ such that $r = sg t$. Hence,
$s^* r = s^* s g t = g t \in G_H$.
Thus $[s^* r \in G_H]$ appearing in (\ref{x5}) is nonzero.
Since $g^* s^* h \in H$, both its source and range projection are in $H^{(0)}$. 
Hence, since $h \in G_H$,
$s^* h$ must be in $G_H$ too in order not to loose information on the source projection of $s^* h$. Since $g \in G_H$, the source projection of $g^*$ and the range projection of $s^* h$ must perfectly fit together such that $g^* s^* h \in H$.
But this implies that $g \cdot(g^* s^* h) = s^* h$ and hence $g \equiv s^* h$.
This implicates $s^* r \equiv g \equiv s^* h$.
We have obtained that $[s^* r \equiv s^* h]\,[s^* h \in G_H]$ appearing in (\ref{x5}) is nonzero.
Hence (\ref{x5}) is nonzero.
In the other direction suppose that (\ref{x5}) is nonzero.
Since $g^* s^* h \in H$ and $g,h \in G_H$, we can completely analogously deduce as before that
$s g \equiv h$. This already implies that $[s g \in G_H]$ appearing in (\ref{x4}) is nonzero.
Since $s^* h, s^* r \in G_H$ and $h,r \in G_H$, we get $s s^* \ge h h^*, r r^*$.
Hence $s^* r \equiv s^* h$ implies $r \equiv h \equiv s g$.
Thus (\ref{x4}) is nonzero.
\end{proof}

%



\begin{lemma}	
The inner products of $X_0$ are positive.
\end{lemma}	

\begin{proof}
Let $(a_\alpha)_\alpha$ be an approximate identity of $A$.
Let $x=\sum_{s=1}^m b_s \rtimes h_s$ in $X_0$ and choose for every {\em different} 
equivalence class $h_s H$ in $G_H/H$
exactly one representative $g_i := h_s \in G_H$, where $1 \le i \le n$ with $n \le m$.
(We have a new index $i$ because for $h_{s_1} H= h_{s_2} H$ we would choose a common $g_i$.)
Set $x_{i,\alpha} = a_\alpha \rtimes g_i \in X_0$.
Set $x_\alpha = \sum_{i=1}^n \langle x_{i,\alpha},x_{i,\alpha} \rangle_{E_0} x \in X_0$.
Then
\begin{eqnarray}
x_\alpha &=&  
\sum_{i=1}^n \sum_{s=1}^m ( a_\alpha g_i g_i^*(a_\alpha^*) \otimes g_i \rtimes g_i g_i^*)
(b_s \rtimes h_s)  \nonumber  \\ 
&=& \sum_{i,s}\,\, [g_i g_i^* h_s \in G_H] \, [g_i \equiv g_i g_i^* h_s]  \,\, a_\alpha g_i g_i^*(a_\alpha^*)\,
 g_i g_i^*(b_s) \rtimes g_i g_i^* h_s  \nonumber   \\
&=& \sum_{i,s}\,\, [g_i \equiv h_s]  \,\, a_\alpha g_i g_i^*(a_\alpha^*)\,
 g_i g_i^*(b_s) \rtimes g_i g_i^* h_s   \label{x12}  \\
&=& \sum_{s=1}^m \,\, h_s h_s^*( a_\alpha a_\alpha^*)\,
 b_s \rtimes h_s \label{x14} 
\end{eqnarray}
where identity (\ref{x12}) follows from the fact that $g_i g_i^* h_s \in G_H$ and $h_s \in G_H$ implies
$g_i g_i^* \ge h_s$, and so $g_i \equiv g_i g_i^* h_s$ implies $[g_i \equiv h_s] \neq 0$, which, on the other hand,
implies $g_i g_i^* = h_s h_s^*$ and thus $[g_i g_i^* h_s \in G_H] \, [g_i \equiv g_i g_i^* h_s] \neq 0$.
Identity (\ref{x14}) follows because we have chosen for each $h_s$
one but at most one equivalent $g_i$.
We used also (\ref{sc2}) and Lemma \ref{lemma1}.(ii) there.
Also (\ref{sc2}) is used to easily compute that $\langle x, x-x_\alpha\rangle_{B_0} \rightarrow 0$.
%
Consequently,
\begin{eqnarray*}
\langle x, x \rangle_{B_0} &=&
\lim_\alpha \langle x, x_\alpha \rangle_{B_0} 
=
\lim_\alpha  \sum_{i=1}^n \Big \langle x,  \langle x_{i,\alpha},x_{i,\alpha} \rangle_{E_0} x \Big \rangle_{B_0} \\
&=&
 \lim_\alpha \sum_{i=1}^n \langle x,x_{i,\alpha}\rangle_{B_0} \langle x,x_{i,\alpha}\rangle_{B_0}^* \; \ge\; 0
\end{eqnarray*}
as in Green \cite{green1978}, page 202, with the last identity of (\ref{eq8}).
The argument for the positivity of $\langle x, x\rangle_{E_0}$ is similar.
Here we choose, for example, $x_\alpha = x \sum_{e \in H^{(0)}}
\langle a_\alpha \rtimes e, a_\alpha \rtimes e \rangle_{B_0}$.
\end{proof}



\begin{proposition}
We have the inequalities
\begin{equation}   \label{normineq}
\langle f x, f x \rangle_{B_0} \le \|f\|^2_{E_0}  \langle x, x \rangle_{B_0}, 
\quad
\langle x b, x b \rangle_{E_0} \le \|b\|^2_{B_0}  \langle x, x \rangle_{E_0}
\end{equation}
for all $x \in X_0$, $f \in E_0$ and $b \in B_0$ (cf. \cite[Def. 6.10]{rieffel1974}).
\end{proposition}

\begin{proof}
For the proof of the first inequality regard $X_0$ as an pre-Hilbert $B_0$-module.
Let $\calm(A)$ denote the multiplier algebra of $A$.
For a nonzero standard element $f=a \otimes r \rtimes s \in E_0$ 
we compute $f^* f$ with the formulas of Lemma \ref{lemma1}.(iii),	
and for another $x \in X_0$ we obtain, with Proposition \ref{straightprop}, 
\begin{eqnarray*}
&& \|f \|_{E_0}^2 \langle x , x \rangle_{B_0} - \langle f x , f x \rangle_{B_0}
 \; = \;  \big \langle (\|f \|_{E_0}^2 - f^* f)x , x \big \rangle_{B_0}   \\
&=&  \big \langle \big(\|f \|_{E_0}^2 - s^*(a^* a ) \otimes s^* r \rtimes s^* s \big) x , x \big \rangle_{B_0}   \\
&=&   \big \langle  z x , z x \big \rangle_{B_0}  + \big \langle  (1-p)x ,(1-p) x \big \rangle_{B_0}
 \ge 0 ,
\end{eqnarray*}
where
$z := \big (\|f \|_{E_0}^2 - s^*(a^* a ) \big )^{1/2} \otimes s^* r \rtimes s^* s$ and
$p := \|f \|_{E_0}^2 \otimes s^* r \rtimes s^* s$ are elements in $\big (\calm(A) \otimes C_0(G_H/H) \big) \rtimes G$.
Of course, we had here temporarily to replace our coefficient algebra $A$ by $\calm(A)$ in order to include
$\|f\|_{E_0}$ and have therefore some slightly larger new $E_0$.
(Note that for general $f \in E_0$, $(\|f \|_{E_0}^2 - f^* f)^{1/2}$ need not be in $E_0$ and that is why we need
to consider elementary elements $f$.)
%

By applying the norm $\|\cdot\|_{B_0}$ in $B_0$ to this inequality, 
we obtain
$\|f x \| \le \|f\|_{E_0} \, \|x\|$ (where $\|x\|:= \|\langle x,x\rangle_{B_0}\|^{1/2}$) for
such elementary elements $f \in E_0$, and by taking sums of such elements
we readily obtain $\|f x \| \le \|f\|_{\ell^1 (G,A \otimes C_0(G_H/H) )} \, \|x\|$ for all $f \in E_0$.
Hence, the $E_0$-module multiplication on $X_0$ is a $*$-homomorphism $E_0 \rightarrow \call(X_0)$ which is an $\ell^1$-contractive representation into a pre-$C^*$-algebra.
Since by definition the $C^*$-norm closure of $E_0$ is the enveloping $C^*$-algebra of $\ell^1(G,A \otimes C_0(G_H/H) )$ (cf. \cite{khoshkamskandalis2004}) and so induced by the sum over all $\ell^1$-contractive representations,
we must get $\|f\|_{\call(X_0)} \le \|f\|_{E_0}$.
It is well known from the theory of Hilbert-modules that one has $\langle f x, f x \rangle_{B_0} \le \|f\|^2_{\call(X_0)}  \langle x, x \rangle_{B_0}$ for adjoint-able operators $f$ (see for instance Lance \cite{lance}, Prop. 1.2), and hence the first inequality
of (\ref{normineq}).
The second inequality of (\ref{normineq}) is proved similarly (but is easier as $B_0$ is even norm-closed).
\end{proof}

%




\begin{definition}
{\rm
Denote by $E_X \subseteq \overline{E_0}$ the norm closure of $\langle X_0,X_0\rangle_{E_0}$
under the $C^*$-norm $\|\cdot\|_{E_0}$, and by $B_X \subseteq \overline{B_0}$
the norm closure of $\langle X_0,X_0\rangle_{B_0}$ under the $C^*$-norm $\|\cdot\|_{B_0}$.
We now apply the argument following \cite[Prop. 3.1]{rieffel1979}
to see that $X_0$ may be completed in semi-norm $\|x\| = \|\langle x,x\rangle_{B_0}\|^{1/2}$
(after factoring out the elements of norm $0$) to obtain an $E_X-B_X$ imprimitivity bimodule $X$.
}
\end{definition}

%
%

\begin{theorem}    \label{theoremImprimitivity}
Let $H'$ be a finite sub-inverse semigroup of an inverse semigroup $G$
and denote by $H$ its associated finite groupoid (Definition \ref{defH}).
Let $A$ be a $G$-algebra.
Then we have a $C^*$-algebraic Morita equivalence
$$C_0(G_H/H,A) \widehat \rtimes G \quad \sim_M  \quad  A \widehat \rtimes H'$$
via the $E_X-B_X$ imprimitivity bimodule $X$ and isomorphisms
$E_X \cong C_0(G_H/H,A) \widehat \rtimes G$ and $B_X \cong A \widehat \rtimes H'$.
%
\end{theorem}

\begin{proof}
The finite dimensional $C^*$-algebra $B_X = B_0$ is canonically isomorphic to the groupoid crossed product $A \rtimes H$, which is canonically isomorphic to the inverse semigroup crossed product $A \widehat \rtimes H'$ by \cite[Thm. 7.2]{0992.46051}.
To meet exactly the assumptions in \cite{0992.46051}, switch 
to the {\em carrier algebra} $\tilde A = p(A)$ for $p = \sum_{e \in H^{(0)}} e$ of $A$,
which does not change the crossed product, that is,
$\tilde A \widehat \rtimes H' = A \widehat \rtimes H'$. 

Denote by $C_0(G_H/H,A) \subseteq  A \otimes C_0(G_H/H)$
the norm closure of the linear span of
$$\{a \otimes r \in A \otimes C_0(G_H/H)|\, a \in A_{r r^*}, r \in G_H/H\}.$$
Note that $C_0(G_H/H, A)$ is a $G$-invariant ideal in $A \otimes C_0(G_H/H)$ and so
\begin{equation}  \label{x15}
C_0(G_H/H,A) \widehat \rtimes G \; \subseteq \; (A \otimes C_0(G_H/H)) \widehat \rtimes G
\end{equation}
embeds by Lemma \ref{lemma2}.
%
Using (\ref{x15}), let 
\begin{equation}   \label{x17}
\sigma: E_X \longrightarrow C_0(G_H/H,A) \widehat \rtimes G \;:\; \sigma(a \otimes r \rtimes g) = a \otimes r \widehat
\rtimes g
\end{equation}
be the canonical map ($a \in A, r \in G_H/H$ and $g \in G$). (Note that there is always a canonical map $A \rtimes G \rightarrow A \widehat \rtimes G$.)
%
It is surjective, because given a nonzero elementary element $a a^* \otimes r \widehat \rtimes g$
in $C_0(G_H/H,A) \widehat \rtimes G$
with $a \in A_{r r^*}, r \in G_H/H$ and $g \in G$
we note that
$$a a^* \otimes r \widehat \rtimes g = r r^*(a a^* \otimes r) \widehat \rtimes g g^* g
= g g^* (a a^*) \otimes g g^* r \widehat \rtimes  r r^* g
$$
by the permeability (compatibility) of $\widehat \rtimes$ for projections in $E(G)$, see (\ref{compat}),
so that we may assume that the given element
$a a^* \otimes r \widehat \rtimes g$ satisfies $a \in A_{r r^*}, r \in G_H/H$ and $g \in G_E$
with $r r^* = g g^*$.
Hence we get
%
%
$$a a^* \otimes r \widehat \rtimes g
=  \sigma \big ( \langle a \rtimes r, g^*(a) \rtimes g^* r \rangle_{E_0} \big ).$$


If $\sigma$ were not injective, then its kernel $J \subseteq E_X$ were nonzero, and so would correspond to a nonzero ideal
$I$ in $B_X$ via the $E_X-B_X$ imprimitivity module $X$ (see \cite[Cor. 3.1]{rieffel1979}), which then would contain a nonzero element of the form $a \rtimes e \in I$ with $e \in H^{(0)}$.
A nonzero element of the form $f=\langle a \rtimes e,a \rtimes e \rangle_{E_0}$ would be in $J$, however $\sigma$ is nonzero on $f$.
We have obtained our result. 
\end{proof}

\begin{definition}
{\rm
Now assume that $D$ is a $H'$-algebra. Define, similarly as in \cite[\S 5 Def. 2]{kasparov1995},
\begin{eqnarray*}
\mbox{Ind}_{H'}^G(D)  &:=&  \{ f: G_H \rightarrow D \, | \,\, \forall g \in G_H, t \in H \mbox{ with } gt \in G_H: f(g t) = t^*(f(g)),\\ 
&& \qquad \|f(g)\| \rightarrow 0 \mbox{ for } g H \rightarrow \infty \mbox{ in } G_H/H \, \}.
\end{eqnarray*}
It is a $C^*$-algebra under the pointwise operations and the supremum's norm and becomes a $G$-algebra under
the $G$-action
$(g f)(h) := [g^* h \in G_H ] \,\,f ( g^* h)$
for $g \in G$, $h \in G_H$ and $f \in \mbox{Ind}_{H'}^G(D)$. 
}
\end{definition}

\begin{corollary}  \label{greenImp2}
Let $H' \subseteq G$ be a finite sub-inverse semigroup of an inverse semigroup $G$. 
Let $D$ be a $H'$-algebra.
Then we have a $C^*$-algebraic Morita equivalence
$$\mbox{Ind}_{H'}^G(D) \widehat \rtimes G \quad \sim_M \quad D \widehat \rtimes H'.$$
\end{corollary}

\begin{proof}
Let $A$ denote $\mbox{Ind}_{H'}^G(D)$.
Consider the $H'$-invariant ideal $A_0$ of $A$ consisting of all functions which 
vanish outside $H$.
Let us again view $B_X$ as $B_X = A \rtimes H$ as in the last proof before.
Then $A_0 \rtimes H$ embeds canonically as an ideal $J$ in $A \rtimes H=B_X$, and
by \cite[Cor. 3.1]{rieffel1979}, associated to $J$ is the submodule in $X$
generated by
$$Y_0= \{y \in X_0|\, \langle y,y\rangle_{B_X} \in J\} = \mbox{span}\{ g(a) \rtimes g \in X_0 | \,a \in A_0, g \in G_H\}  \; \subseteq \; X,$$
and the ideal $I$
in $E_X$
generated by (and actually norm closure of)
$$\langle Y_0,Y_0\rangle_{E_X}
= \mbox{span}\{ g (a) \otimes g \widehat \rtimes g h^* \in E_X| \, a\in A_0, \, g,h \in G_H\} \; \subseteq \; E_X.$$

Note that
we are identifying
$E_X \cong C_0(G_H/H,A) \widehat \rtimes G$ 
by the isomorphism $\sigma$ stated in (\ref{x17}).
Using Lemma \ref{lemma2},
the ideal $I$ is canonically 
isomorphic to $K \widehat \rtimes G  \subseteq
E_X$, where $K$ denotes the $G$-invariant
ideal in $C_0(G_H/H,A)$
which is the norm closure of the linear span of
$$\{ g(a) \otimes g \in C_0(G_H/H,A) | \,a \in A_0, g \in G_H/H\}.$$
To see that the 
identical embedding
$I \rightarrow K \widehat \rtimes G$ is surjective,
write a given nonzero element $g(a) \otimes g \widehat \rtimes s \in K \widehat \rtimes G$
($a \in A_0, g \in G_H/H$ and $s \in G$)
as
$$s s^* g(a) \otimes s s^* g \widehat \rtimes s s^* g g^* s \quad \in I$$
with $s s^*g, s^* g \in G_H$
by the compatibility of $\widehat \rtimes$, see (\ref{compat}).
%
%
%
%
%
%
%
%
%
We have a $G$-equivariant isomorphism $\psi:A \rightarrow K$
defined by
$$\psi(f) = \sum_{g \in G_H/H} 
f|_{g} \otimes g = \sum_{g \in G_H/H} g \big (g^*(f)|_{g^* g} \big) \otimes g \quad\in K,$$
where $f \in A = \mbox{Ind}_{H'}^G(D)$ and
$f|_{g}\in A$ denotes the function $f|_{g}(k)= [k \equiv g]f(k)$ for all $k \in G_H$.
%

There is a $H'$-equivariant epimorphism
$\Phi: D \rightarrow A_0$ given by $\Phi  (d )(t) = t^* (d)$
for $t \in H$ and $d \in D$. It is an isomorphism on the carrier algebra of $D$, so that $D \widehat \rtimes H' \cong A_0 \widehat \rtimes H' \cong A_0 \rtimes H \cong J$.
Consequently we have obtained, by restricting to the ideals $I$ and $J$ in Theorem \ref{theoremImprimitivity}
and applying \cite[Cor. 3.1]{rieffel1979}, our result.
\end{proof}


{\bf Acknowledgement.}
We thank the 
the Universidade Federal de Santa Catarina in Florian\'opolis
for the support we received when developing the content
of this paper in 2014. 
This note presents a slightly modified version (more detailed proofs)
of a preprint in arXiv
from 2014.

\bibliographystyle{plain}
\bibliography{references}

\end{document}